\documentclass[11pt,reqno,a4paper]{amsart}

\usepackage{amsmath,amssymb,amsthm,amsaddr}
\usepackage[utf8]{inputenc}
\usepackage[british]{babel}
\usepackage{enumitem}

\usepackage{tikz}
\usepackage[framemethod=TikZ]{mdframed}

\usepackage{multirow}
\usepackage{colortbl}

\usepackage{caption}
\usepackage{subcaption}

\usepackage[margin=3cm,top=4cm,bottom=3.5cm,footskip=2cm,headsep=2cm]{geometry}

\newtheorem{theorem}{Theorem}

\theoremstyle{plain}
\newtheorem{lemma}[theorem]{Lemma}

\theoremstyle{definition}

\newtheorem{problem}[theorem]{Problem}

\newtheorem{remark}[theorem]{Remark}
\newtheorem{assumption}[theorem]{Assumption}

\def\eps{\varepsilon}
\def\e{\eps}

\def\dx{\partial_x}
\def\dxx{\partial_{xx}}
\def\dt{\partial_t}

\def\V{\mathcal{V}}
\def\dV{\mathcal{V}_\partial}
\def\E{\mathcal{E}}
\def\G{\mathcal{G}}
\def\T{\mathcal{T}}
\def\dT{\partial\mathcal{T}}

\def\eps{\varepsilon}

\def\Th{\mathcal{T}_h}
\def\Xh{\mathcal{X}_h}

\def\uup{u_h^{up}}
\def\wup{w_h^{up}}

\newcommand{\VVert}[1]{{\left\vert\kern-0.25ex\left\vert\kern-0.25ex\left\vert#1\right\vert\kern-0.25ex\right\vert\kern-0.25ex\right\vert}}

\usepackage{accents}

\usepackage[bookmarksnumbered=true,draft=false]{hyperref}
\hypersetup{
  colorlinks=true,
  allcolors = red,
  linkcolor=blue, 
  menucolor=blue,
  urlcolor = blue,
  frenchlinks=false,
  pdfborder={0 0 0},
  naturalnames=false,
  hypertexnames=false,
  breaklinks
}

\title[Asymptotic preserving hybrid-dG method]{An asymptotic preserving hybrid-dG method for convection-diffusion equations on pipe networks}
\author{H. Egger$^{*,\dag}$ \and N. Philippi$^\dag$}
\address{$^*$Institute for Numerical Mathematics, Johannes-Kepler University Linz, Austria \\
$^\dag$Johann Radon Institute for Computational and Applied Mathematics, Linz, Austria}
\email{herbert.egger@jku.at}
\email{nora.philippi@ricam.oeaw.ac.at}

\usepackage{caption}
\begin{document}

\begin{abstract}
We study the numerical approximation of singularly perturbed convection-diffusion problems on one-dimensional pipe networks. In the vanishing diffusion limit, the number and type of boundary conditions and coupling conditions at network junctions changes, which gives rise to singular layers at the outflow boundaries of the pipes. A hybrid discontinuous Galerkin method is proposed, which provides a natural upwind mechanism for the convection-dominated case. Moreover, the method automatically handles the variable coupling and boundary conditions in the vanishing diffusion limit, leading to an asymptotic-preserving scheme. A detailed analysis of the singularities of the solution and the discretization error is presented, and an adaptive strategy is proposed, leading to order optimal error estimates that hold uniformly in the singular perturbation limit. The theoretical results are confirmed by numerical tests.
\end{abstract}

\maketitle

\vspace*{-1em}

\begin{quote}
\end{quote}

\begin{quote}
\noindent
{\small {\bf AMS-classification (2000):}
35B25, 
35B40, 
35K20, 
35R02, 
65N30, 
76R99 
}
\end{quote}









\section{Introduction} \label{sec:1}

We are interested in the numerical solution of convection-diffusion processes on one-dimensional pipe networks. Such problems describe, e.g., the contaminant transport in water supply networks and systems of 1D-cracks \cite{Laird05,Oppenheimer00}, or the distribution of energy in district heating networks \cite{Hauschild20}.
Related nonlinear problems have been studied in the context of traffic flow modelling; see \cite{Garavello06} for an introduction. 

\textbf{Problem setting and asymptotic analysis.}
%
On every single pipe of the network, the transport of matter is described by the convection-diffusion problem
\begin{alignat}{2}
    \dt u^\eps + b \dx u^\eps &= \eps \dxx u^\eps, \qquad && x \in (0,\ell), \ t>0,  \label{eq:1}\\
    u^\eps &= \hat g, \qquad && x \in \{0,\ell\}, \ t>0. \label{eq:2}
\end{alignat}
Here $u^\eps$ denotes the quantity of interest, e.g., the concentration of the contaminant, $b>0$ is the flow velocity, $\eps>0$ the diffusion coefficient, and $\hat g$ suitable boundary data.
%
In the vanishing diffusion limit $\eps \to 0$, the boundary condition at $x=\ell$ becomes obsolete. This leads to a boundary layer for $0 < \eps \ll 1$ accompanied by a blow-up of the derivatives of the solution, which can be expressed by
\begin{align} \label{eq:3}
    \|u^\eps\|_{L^\infty(0,t_{max};H^1(0,\ell))} \approx C/\sqrt{\eps}.
\end{align}
In contrast to that, the solution $u^\circ$ of the transport problem, which arises in the limit $\eps=0$, may be perfectly smooth within the pipe, although violating the boundary condition at $x=\ell$. 
Maximum principles allow to verify that both, $u^\eps$ and $u^\circ$, are bounded uniformly, and under appropriate assumptions, one can further show that
\begin{align} \label{eq:4}
    \|u^\eps - u^\circ\|_{L^\infty(0,t_{max};L^2(0,\ell))} \le C \sqrt{\eps}
\end{align}
with a uniform constant $C$; see e.g. \cite{Roos08}.
This asymptotic estimate suggests that the transport solution $u^\circ$ may serve as a good approximation for $u^\eps$ for small $0 < \eps \ll 1$. 

\textbf{Extension to networks.}
Additional coupling conditions are required to model the flow through junctions of more than two pipes and to determine the values $\hat g$ in \eqref{eq:2} at internal junctions.
The number and type of these conditions may change in the singular limit $\eps \to 0$, which gives rise to additional interior layers.
As elaborated in \cite{EggerCD20}, the blow-up of the solution and the asymptotic estimates for a single pipe, however, carry over almost verbatim to the network setting.
We refer to \cite{Guarguaglini21} for results concerning a general class of coupling conditions and to \cite{Barcena21} for related investigations concerning optimal control. Vanishing diffusion limits for scalar conservation laws in traffic networks are studied in \cite{Coclite10}.

\textbf{Numerical approximation.}
The discretization of singularly perturbed convection-diffusion problems is well studied in the literature; see \cite{Roos08} for an comprehensive survey. 
A key ingredient for the robust approximation is the use of layer-adapted 
meshes. Finite element approximations on Shishkin-type meshes were investigated in \cite{Constantinou15, Duran06, Roos97}; also see~\cite{Roos15} for the analysis of higher order schemes and \cite{Singh20,Xie10} for the investigation of discontinuous Galerkin methods on various layer-adapted meshes.
In this paper, we consider the spatial approximation by a \emph{hybrid discontinuous Galerkin} (dG) method; see \cite{Cockburn09,Ern11} for background material. 
Related methods for problems on multi-dimensional domains were investigated in \cite{Egger10, Fu15, Nguyen11} and by \cite{Chen19} in the context of optimal control.

\textbf{Main contributions.}
The proposed hybrid-dG method not only provides an upwind mechanism for handling the convection-dominated regime but, more importantly, allows a natural treatment of the coupling conditions at pipe junctions. 
In the singular limit $\eps = 0$, the method automatically reduces to the hybrid-dG scheme for pure transport on networks, which has been analysed in~\cite{Egger20}.
Hence, the method is formally \emph{asymptotic-preserving} in the vanishing diffusion limit.
By extension of previous results, we will establish uniform error estimates, 
which for a single pipe and approximations of order $k$ take the form
\begin{align} \label{eq:6}
\|u^\eps - \tilde u_h^\eps\|_{L^\infty(0,T;L^2(0,\ell))} 
\le C \max(h^{k+1},\min(\sqrt{\eps},h^k)).
\end{align}
The constant $C$ here only depends on the regularity of the boundary data, but is independent of $\eps$ and $h$. 
The approximation $\tilde u_h^\eps$ will be chosen adaptively as 
\begin{align*}
    \tilde u_h^\eps = \begin{cases} u_h^\eps, & \eps \ge h^{2k}, \\
    u_h^\circ, & \eps < h^{2k},\end{cases}
\end{align*}
where $u_h^\eps$ is the hybrid-dG approximation for the convection-diffusion problem on a layer-adapted mesh $\Th^\eps$ of Gartland-type \cite{Gartland88}, while $u_h^\circ$ is the hybrid-dG approximation \cite{Egger20} for the pure transport problem on a uniform mesh $\Th^\circ$. 
Following \cite{Gartland88,Roos97}, the mesh $\Th^\eps$ is chosen uniformly in the interval $(0,x^*(\eps))$ away from the layer, and geometrically refined within the layer $(x^*(\eps),\ell)$. As  transition point, we will chose
\begin{align}
x^*(\eps) \approx \eps \log(1/\eps) \lesssim \eps\log(1/h);
\end{align}
the second inequality only holds due to the condition $\eps \ge h^{2k}$. 
Using this observation, one can show that the number of elements in $\Th^\eps$ is of optimal order, i.e. $N \approx h^{-1}$. At the same time, the choice of the transition point $x^*(\eps)$ simplifies the convergence analysis significantly; see Section~\ref{sec:proof_main}.
Our results cover the case of a single pipe as well as finite networks of pipes with rather general topology, in particular, including cycles.  
Similar uniform convergence estimates also hold for fully discrete schemes obtained after time discretization by appropriate time stepping schemes, which will be demonstrated in numerical tests.

\textbf{Outline.}
In Section~\ref{sec:prob}, we introduce our notation and main assumptions, state the convection-diffusion and pure transport problems on networks, and then summarize some important properties of their solutions. 
The hybrid-dG method is proposed in Section~\ref{sec:HDG}, and we state a preliminary error estimate.
Section \ref{sec:main} contains our main convergence result and its proof. For completeness of the presentation, the proofs of some auxiliary technical results are included in the appendix.
For illustration of our theoretical results, we give numerical tests in Section~\ref{sec:num}.
The presentation closes with a short summary and remarks concerning possible extensions of our results.

\section{Problem statement} \label{sec:prob}

Let us start by introducing the relevant notation and then give a complete definition of the problems under consideration. After that, we state the main assumptions for our analysis and summarize the basic properties of solutions to the continuous problems.

\subsection{Notation}\label{subsec:not}

Following \cite{Egger20,Mugnolo14}, the network topology is described by a finite, directed and connected graph $\G=(\V,\E)$, with vertices $\V=\{v_1,\dots,v_n\}$ and edges $\E=\{e_1,\dots,e_m\} \subset \V \times \V$. 
We write $\E(v)=\{e \in \E: e=(v,\cdot) \ \text{or} \ e=(\cdot,v)\}$ for the set of edges incident to a vertex $v\in\V$, and $\dV=\{v\in\V:|\E(v)|=1\}$,  $\V_0=\V\backslash\dV$ for the sets of boundary and internal vertices. 
As usual, $|S|$ describes the cardinality of a finite set $S$. 
For any edge $e=(v^{in},v^{out})$, we define two values
\begin{align*}
n_{e}(v^{in}):=-1\quad\text{and}\quad n_{e}(v^{out}):=1
\end{align*} 
indicating the start and end point of the edge, and we set $n_e(v):=0$ for $v\in\V\backslash\{v^{in},v^{out}\}$.
We write $\mathcal{E}^{in}(v):=\{e\in\mathcal{E}: n_{e}(v)>0\}$ and $\mathcal{E}^{out}(v):=\{e\in\mathcal{E}:n_{e}(v)<0\}$ for the sets of edges pointing into and out of the vertex $v \in \V$, respectively. 
Furthermore, we split $\V_\partial$ into a set of boundary vertices $\mathcal{V}_{\partial}^{in}:=\{v\in\mathcal{V}_{\partial}:n_{e}(v)<0\ \text{for}\ e\in\mathcal{E}(v)\}$, from which edges leave into the network, and the complement $\mathcal{V}_{\partial}^{out}:=\{v\in\mathcal{V}_{\partial}:b_{e}n_{e}(v)>0\ \text{for}\ e\in\mathcal{E}(v)\}$, in which edges terminate; see Figure~\ref{fig:topology} for an illustration. 
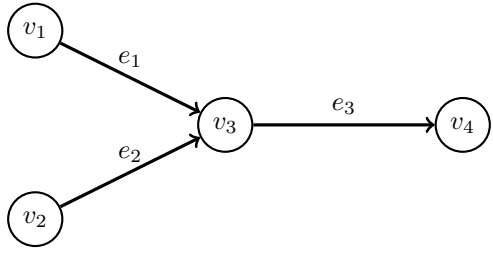
\begin{figure}[ht]
\centering
\small
\begin{tikzpicture}[scale=2.5]
\node (A) at (0,0.5) [circle,draw,thick] {$v_1$};
\node (B) at (0,-0.5) [circle,draw,thick] {$v_2$};
\node (C) at (1,0) [circle,draw,thick] {$v_3$};
\node (D) at (2.25,0) [circle,draw,thick] {$v_4$};

\draw[->, very thick] (A) to node[above] {$e_1$} (C);
\draw[->, very thick] (B) to node[above] {$e_2$} (C);
\draw[->, very thick] (C) to node[above] {$e_3$} (D);
\end{tikzpicture}
\caption{Network with edges $e_1=(v_1,v_3)$, $e_2=(v_2,v_3)$, and $e_3=(v_3,v_4)$, inner vertex $\V_0=\{v_3\}$, and boundary vertices $\V_\partial=\{v_1,v_2,v_4\}$.
The set $\mathcal{E}(v_3)=\{e_1,e_2,e_3\}$ contains all edges incident to the vertex $v_3$. This set can be split into $\E^{in}(v_3)=\{e_1,e_2\}$ and $\E^{out}(v_3)=\{e_3\}$ with the edges that point into or out of the vertex $v_3$, respectively.
The boundary vertices are split into $\V_{\partial}^{in}=\{v_1,v_2\}$ and $\V_\partial^{out}=\{v_4\}$ containing the vertices from which edges originate or in which edges terminate. 
}
\label{fig:topology}
\end{figure}
%
To any edge $e \in \E$, we associate a length $\ell_e$, and identify $e \simeq (0,\ell_e)$ with an interval.
By $L^2(e)=L^2(0,\ell_e)$ and 
\begin{align*}
L^2(\E)=L^2(e_1)\times\dots\times L^2(e_m)=\{u: u_e\in L^2(e)\ \text{for all}\ e\in\E\},
\end{align*}
we designate the spaces of square integrable functions on a pipe $e$ and the network $\E$, respectively, and we write $u_e=u\vert_e$ for the restriction of $u$ to the edge $e$. 
The norm and scalar product for the space $L^2(\E)$ are given by
\begin{align*}
\Vert u\Vert_{L^2(\E)}^2=\sum\nolimits_{e\in\E}\Vert u_{e}\Vert_{L^2(e)}^2\qquad\text{and}\qquad (u,w)_{L^2(\E)}=\sum\nolimits_{e\in\E}(u_{e},w_{e})_{L^2(e)}.
\end{align*}
We further define the broken Sobolev spaces
\begin{align*}
H_{pw}^k(\E)=\{u\in L^2(\E): u_{e}\in H^k(e)\ \text{for all}\ e\in\E\}, \qquad k \ge 0.
\end{align*}
Note that $H^0_{pw}(\E)=L^2(\E)$ and for $k \ge 1$ the functions $u\in H^k_{pw}(\E)$ are continuous along edges $e\in\E$, but may be discontinuous across junctions $v\in\V_0$. 
The sub-space of functions that are continuous also across junctions is denoted by $H^1(\E)$. 
Any $u\in H^1(\E)$ has unique values $u(v)$ for every $v\in\V$, and we write $\ell_2(\V)$ for the space of possible vertex values.

\subsection{Convection-diffusion problem}\label{sec:2.2}

We are now in the position to introduce the parabolic problem for $\eps>0$.
Along the pipes of the network, we assume
\begin{alignat}{2}
\dt u^\eps_e(x,t) + b_e\dx u^\eps_e(x,t) &= \eps \dxx u^\eps_e(x,t), \qquad &&x\in e, \ e \in \E,\ t>0.\label{convdiff:1} 
\end{alignat}
Like on a single pipe, we enforce Dirichlet conditions at the boundary vertices, i.e.,
\begin{alignat}{2}
u^\eps(v,t)&=\hat g_v(t), \qquad &&v\in\dV,\  t>0.\label{convdiff:2} 
\end{alignat}
At the interior vertices $v \in \V_0$, on the other hand, we require the coupling conditions
\begin{alignat}{2}
u^\eps(v,t) &= \hat{u}^\eps_v(t) , \qquad &&v\in\V_0,\ t>0,\label{convdiff:3} \\
\sum\nolimits_{e\in\mathcal{E}(v)}\big(b_e u^\eps_e(v,t)-\e \dx u^\eps_e(v,t)\big)n_e(v)&=0, \qquad &&  v\in\mathcal{V}_0,\ t>0, \label{convdiff:4}
\end{alignat}
which encode continuity of the density $u^\eps$ and conservation of mass across network junctions. 
The conditions \eqref{convdiff:3} and \eqref{convdiff:4} make up $|\E(v)|+1$ coupling conditions at each interior vertex $v\in\V_0$, corresponding to the number of all incident edges and the additional unknown vertex value $\hat u_v^\eps$, which is called \emph{hybrid variable} in the following.

\subsection{Limiting transport problem}

In the vanishing diffusion limit $\eps=0$, the flow on the pipes is described by the hyperbolic transport equation
\begin{alignat}{2}
\dt u_e^\circ(x,t)+b_e\dx u_e^\circ(x,t) &=0, \qquad && x\in e,\ e \in \E, \  t>0.\label{transp:1} 
\end{alignat}
We can now prescribe Dirichlet data only at the inflow boundary vertices, i.e., 
\begin{alignat}{2}
u_e^\circ(v,t) &= \hat{g}_v(t), \qquad \qquad && v=\dV^{in}, \ t>0.\label{transp:2}
\end{alignat}
The coupling across network junctions is further described by
\begin{alignat}{2}
u_e^\circ(v,t) &= \hat u_v^\circ(t), \qquad \qquad  \qquad \qquad v\in \V_0, \ e \in \E^{out}(v), \ t>0, \label{transp:3} \\
\sum\nolimits_{e\in\E^{in}(v)}b_e\hat{u}_v^\circ(t)n_e(v)&=\sum\nolimits_{e\in\E^{in}(v)}b_e u_e^\circ(v,t)n_e(v), \qquad  v\in\V_0 \cup \dV^{out}, \ t>0.\label{transp:4}
\end{alignat}
Condition \eqref{transp:3} fixes the densities at the inflow vertices of the pipes $e \in \E^{out}(v)$ to the vertex value $\hat u_v^\circ$, which is determined by the mixing rule \eqref{transp:4} as a convex combination of the values $u_e^\circ(v)$ coming from the edges  $e\in\E^{in}(v)$ pointing into the vertex $v$.
In summary, this makes up $|\E^{out}(v)|+1$ coupling conditions which determine the values $u_e^\circ(v)$ for the edges $e \in \E^{out}(v)$ originating from $v$ and the vertex value $\hat{u}_v^\circ$. 
Let us emphasise that the number and type of coupling conditions is different from the parabolic case $\eps>0$ above, leading to additional interior layers for vanishing diffusion $\eps\rightarrow 0$; see \cite{EggerCD20} and below.

\subsection{Basic assumption and preliminary results}

For the rest of the presentation, we make use of the following assumptions on the problem data.
\begin{assumption}\label{ass:1}
Let $0 < \eps \le 1$ and $0 < \underline b \le b_e \le \bar b$ for all $e\in\E$
as well as
\begin{align} 
\sum\nolimits_{e\in\E(v)}b_en_e(v)=0, \qquad v\in\V_0.\label{eq:b} 
\end{align}
Furthermore, the boundary data in \eqref{convdiff:3} and \eqref{transp:3} shall satisfy $\hat g\in C^{m+2}(0,t_{max};\ell^2(\dV))$ for some time horizon $t_{max}>0$, with $\dt^n \hat g(0)=0$ for $0\le n\le m$ and some $m\ge0$.
\end{assumption}
The assumptions on $b$ characterize a steady background flow which, for ease of notation, is aligned with the orientation of the edges.
Condition \eqref{eq:b} together with the coupling conditions ensures conservation of mass at interior vertices. 
The boundary data are consistent with trivial initial conditions $u^\eps(0)=0$, and hence the occurrence of initial layers is avoided. 
For later reference, we summarize some basic results about solvability and regularity of solutions for our two model problems. 

\begin{lemma}\label{lem:deriv}
Let Assumption \ref{ass:1} hold. 
Then for any $\eps>0$, the convection--diffusion problem \eqref{convdiff:1}--\eqref{convdiff:4} has a unique solution $(u^\eps,\hat u^\eps)$ with initial value $u^\eps(0)=0$, and 
\begin{align} \label{eq:regularity}
    u^\eps\in C^{m+1}(L^2(\E))\cap C^0(H_{pw}^{2m+2}(\E)),\qquad \hat u^\eps \in C^{m+1}(\ell_2(\V_0)),
\end{align}   
and the derivatives of $u^\eps$ are bounded by
\begin{align}\label{eq:deriv}
|\dt^n\dx^j u^\eps_e(x,t)| \leq C\, (1+\e^{-j} e^{-b_e(\ell_e-x)/\eps})
\end{align}
for all $x\in (0,\ell_e)$, $e\in\E$, $t>0$ and $n\le m,\, j\le 2(m-n)+1$.
Furthermore, also the transport problem \eqref{transp:1}--\eqref{transp:4}
has a unique solution with initial value $u^\circ(0)=0$, and 
\begin{align}\label{eq:regularity_transport}
    u^\circ\in C^{m+1}(L^2(\E))\cap C^0(H_{pw}^{m+1}(\E)),\qquad \hat u^\circ \in C^{m+1}(\ell_2(\V_0)),
\end{align}
and the asymptotic estimate
\begin{align}\label{eq:asympt_estimate_network}
    \|u^\eps - u^\circ\|_{L^\infty(0,t_{max};L^2(\E))} \le C' \sqrt{\eps}
\end{align}
holds true.
The constants $C$, $C'$ only depend on the bounds in Assumption \ref{ass:1}. 
Here $C^{m}(X)=C^{m}([0;t_{max}];X)$ is the space of smooth functions on $[0,t_{max}]$ with values in $X$, and $t_{max}>0$ the chosen time horizon.
\end{lemma}
%
Existence, uniqueness and regularity of the solutions follow readily by semi-group theory; see \cite{EggerCD20,Mugnolo14} and \cite{Dorn10,Egger20} for details. 
%
%
The asymptotic estimate \eqref{eq:asympt_estimate_network} was proven in \cite{EggerCD20}. 
The bounds for the derivatives can be established with similar arguments as in \cite{Kellogg78, Rao17, Stynes89}. For convenience of the reader, a detailed proof of \eqref{eq:deriv} for the problem on networks is presented in Appendix~\ref{sec:thm:deriv}. 

\section{The hybrid discontinuous Galerkin method} \label{sec:HDG}

We now turn to the discretization of the problems introduced in the previous section. 
In our analysis, we will only consider the semi-discretization in space. In combination with appropriate time-stepping schemes, all results also carry over to fully discrete approximations. 
Corresponding remarks and numerical tests will be presented in Section~\ref{sec:num}. 

\subsection{Mesh and approximation spaces}\label{subsec:HDGnot}

We split every edge $e \simeq (0,\ell_e)$ into appropriate sub-intervals. The global mesh is then given by 
\begin{align*}
\mathcal{T}_h&=\{T_e^i=(x_{e}^{i-1},x_{e}^{i}):i=1,\dots,M_{e}, \ e\in\E\},
\end{align*}
with $0 = x_{e}^0 < x_e^1 < \ldots < x_{e}^{M_{e}}=\ell_{e}$ denoting the mesh points on the edge $e$. 
We write $h_e^i=x^{i}_{e}-x^{i-1}_{e}$ and $h=\max_{e,i}\, h^i_{e}$ for the local and global mesh size, and also use $h_T$ for the size of an element $T\in\T_h$ below.
The extremal points $x_e^0$ and $x_e^{M_e}$ on every edge are identified with vertices $v \in \V$ of the graph $\G(\V,\E)$, and we denote by 
\begin{align*}
\mathcal{X}_h = \{ x_e^i : 0< i < M_e, \ e \in \E\} 
\end{align*}
the remaining mesh points in the interior of the edges. 
Note that the mesh $(\Th,\Xh)$ could also be interpreted as a refinement $\G_h=(\V \cup \Xh, \Th)$ of the original graph $\G=(\V,\E)$.
In accordance with the notation of Section~\ref{sec:prob}, we define broken Sobolev spaces 
\begin{align*}
H_{pw}^k(\mathcal{T}_h)=\{w\in L^2(\E):w|_T\in H^k(T)\ \text{for all}\ T \in\mathcal{T}_h\}.
\end{align*}
For the approximation of solutions to our two model problems, we consider the spaces
\begin{align*}
W_h&=\{w_h\in L^2(\E):w_h|_{T}\in P_k(T)\ \text{for all}\ T\in\mathcal{T}_h\},
\end{align*} 
consisting of all piecewise polynomials of degree $\le k$ over the mesh $\Th$. 
In addition, we will make use of the space of hybrid variables
\begin{align*}
    \hat{W}_h = \{\hat w_h\in\ell_2(\V\cup \Xh):\, \hat w_h(v)=0\ \forall\, v\in\dV\}
\end{align*}
to represent the vertex values at the interior mesh points.  
For convenience of notation, let us introduce the grid dependent scalar products
\begin{align*}
(u,w)_{\mathcal{T}_h} =\sum\nolimits_{T\in\mathcal{T}_h}(u,w)_{L^2(T)},\qquad 
\langle u,w\rangle_{\partial\!\T_h} =\sum\nolimits_{T_e^i\in\T_h} u(x^{i-1}_e)w(x^{i-1}_e)+u(x^{i}_e)w(x^{i}_e), 
\end{align*}
as well as the associated norms $\| w\|_{\mathcal{T}_h}^2=(w,w)_{\mathcal{T}_h}$ and $| w|_{\partial\!\mathcal{T}_h}^2=\langle w,w\rangle_{\partial\!\mathcal{T}_h}$. 

\subsection{An asymptotic preserving discretization method}

For the numerical approximation of solutions to the convection--diffusion problem on networks \eqref{convdiff:1}--\eqref{convdiff:4}, as well as the limiting transport problem \eqref{transp:1}--\eqref{transp:4}, we consider the following discretization scheme. 
\begin{problem} \label{prob:convdiff} 
Let $W_h$ and $\hat W_h$ be defined as above with polynomial degree $k\ge 1$ fixed. 
Find $u_h^\eps\in C^1([0,t_{max}]; W_h)$ with $u_h^\eps(0)=0$ and $\hat{u}_h^\eps\in C^0([0,t_{max}];\hat{W}_h)$, such that  
\begin{align}\label{hdg:convdiff}
(\dt u_h^\eps(t),w_h)_{\T_h}+b_h(u_h^\eps(t),\hat{u}_h^\eps(t);w_h,\hat{w}_h)
+\eps d_h(u_h^\eps(t),\hat{u}_h^\eps(t);w_h,\hat{w}_h)
&= \ell_h^\eps(t;w_h) 
\end{align}	
for all $w_h\in W_h$, $\hat{w}_h\in\hat{W}_h$, and $0 \le t \le t_{max}$,
with bilinear and linear forms defined by
\begin{align}
b_h(u_h^\eps,\hat{u}_h^\eps;w_h,\hat{w}_h) =& - (bu_h^\eps,\dx w_h)_{\T_h}+\langle n b \, \uup,w_h-\hat{w}_h\rangle_{\dT_h},\label{def:bh}\\
d_h(u_h^\eps,\hat{u}_h^\eps;w_h,\hat{w}_h) =&\ (\dx u_h^\eps,\dx w_h)_{\T_h}
- \langle n \dx u_h^\eps,w_h-\hat{w}_h\rangle_{\dT_h} \label{def:dh}\\
&\quad\ + \langle n (u_h^\eps-\hat{u}_h^\eps),\dx w_h\rangle_{\dT_h} 
+ \langle\tfrac{\alpha}{h_{loc}}(u_h^\eps-\hat{u}_h^\eps),w_h-\hat{w}_h\rangle_{\partial \T_h},\nonumber\\
\ell_h^\eps(t;w_h) =& -\langle nb \hat g(t), w_h\rangle_{\dV^{in}} - \langle n \eps \hat g(t), \dx w_h  \rangle_{\dV} + \langle \tfrac{\alpha \eps}{h_{loc}} \hat g(t), w_h \rangle_{\dV}.
\end{align} 
Here $n b\, \uup=\max(n b,0) u_h^\eps + \min(n b,0) \hat{u}_h^\eps$ denotes the convective upwind flux at the vertices, $h_{loc}|_T=h_T$ for $T\in\T_h$, and $\alpha>0$ is a stabilization parameter for the diffusive jump terms; see \cite{Egger10} for a similar definition of the upwind value $\uup$.
\end{problem}
%
Using standard arguments, see e.g. \cite{Ern11,Thomee07}, one can obtain the following local error estimate. For convenience of the reader, a complete proof is provided in Appendix~\ref{sec:proof_prelim}. 
\begin{lemma}\label{lemma:prelim_error}
Let Assumption \ref{ass:1} hold and $(u^\eps,\hat u^\eps)$ be the solution of \eqref{convdiff:1}--\eqref{convdiff:4} with initial value $u^\eps(0)=0$.
Further, let $(u_h^\eps,\hat u_h^\eps)$ be the corresponding solution of Problem~\ref{prob:convdiff}.
Then
\begin{align}\label{eq:prelim_error}
\|u^\eps(t)-u_h^\eps(t)\|_{L^2(\E)}^2\le C\sum\nolimits_{T\in\T_h}(\eps h_T^{2k} + h_T^{2k+2}) \|u^\e\|_{H^1(0,t_{max};H^{k+1}(T))}^2
\end{align}
for all $0<t<t_{max}$ with constant $C$ independent of $\eps$ and $\Th$.
\end{lemma}
\begin{remark}
The above scheme falls into the class of \emph{hybridizable dG methods} introduced in \cite{Cockburn09}.
By formally setting $\eps=0$, we obtain a consistent approximation for the limiting transport problem \eqref{transp:1}--\eqref{transp:4}, which was analysed in~\cite{Egger20}. Hence, the scheme is {\it asymptotic preserving} in the vanishing diffusion limit.
The bound \eqref{eq:prelim_error} holds verbatim for $\eps=0$ and yields order optimal estimates $\|u^\circ-u_h^\circ\|_{L^\infty(0,T;L^2(\E)} \le C h^{k+1}$ for the approximation of the limiting transport problem on uniform meshes. 
On such meshes, however, we do not expect convergence for $\eps \to 0$ due to the blow-up of the derivatives of $u^\eps$, see Lemma~\ref{lem:deriv}, leading to degenerate bounds on the right hand side of \eqref{eq:prelim_error}.
\end{remark}

\section{The main result}
\label{sec:main}

In order to deal with the singularities of the solution $u^\eps$ for $\eps \to 0$
, we consider approximations on layer-adapted grids. This will allow us to establish error estimates that are uniform in the asymptotic parameter $\eps$.

\subsection{Construction of the adaptive grids}

When $\eps < h^{2k}$ we will use a quasi-uniform grid $\Th=\Th^\circ$ with mesh size $h_T \approx h$ for all elements $T \in \Th$. 
For $\eps \ge h^{2k}$ 
we construct a layer-adapted grid $\Th^\eps$ as follows:
For every edge $e \in \E$ we define a transition point 
\begin{align} \label{eq:transition}
    x^\ast_e := \ell_e - \tfrac{k+1}{b_e}\eps \log(1/\eps).
\end{align}
On the intervals $[0,x^\ast_e)$ we consider the mesh-points inherited from the uniform mesh $\Th^\circ$, augmented by the transition point $x^\ast_e =: x_e^{M_e^*}$. The corresponding elements are collected in the set~$\Th^{\eps,1}$. By construction, we have $h_T \lesssim  h$ for $T \in \Th^{\eps,1}$ and $|\Th^{\eps,1}| = M_e^* \le C h^{-1}$. 
In the layer region $(x^\ast_e,\ell_e]$, the mesh points are defined recursively by
\begin{align}\label{eq:gradmesh1}
    h_e^i = \eps h e^{b_e(\ell_e-x_e^i)/\eps(k+1)},\qquad x_e^{i-1} = x_e^i - h_e^i, \qquad i \le M_e,
\end{align}
with $x_e^{M_e}=\ell_e$ denoting the outflow vertex of the edge $e \simeq  (0,\ell_e)$.
The index $M_e$ is chosen such that $x_e^{M_e^*+1}$ is the last point in this sequence which is strictly larger than the transition point $x^\ast_e$; see Figure~\ref{fig:meshes}.
\begin{figure}[ht]
\centering
\begin{tikzpicture}[scale=2,label distance=3mm]
    
    \node (A) at (4.5,0) {$\T_h^\circ$};
    
    \coordinate[label = above: $0$] (B) at (0,0);
    \coordinate (C) at (0.5,0);
    \coordinate (D) at (1.25,0);
    \coordinate (E) at (1.9,0);
    \coordinate (F) at (2.5,0);
    \coordinate[label = above: $\ell_e$] (G) at (4,0);
    \coordinate (H) at (3.25,0);
    
    \draw[|-|,very thick] (B) to (C);
    \draw[-|,very thick] (C) to (D);
    \draw[-|,very thick] (D) to (E);
    \draw[-,very thick] (E) to (F);
    \draw[|-|,very thick] (F) to (H); 
    \draw[-|,very thick] (H) to (G); 

    
    \node (A) at (4.5,-0.5) {$\Th^\eps$};
    
    \coordinate[label = below: $0$] (B) at (0,-0.5);
    \coordinate (C) at (0.5,-0.5);
    \coordinate (D) at (1.25,-0.5);
    \coordinate (E) at (1.9,-0.5);
    \coordinate (F) at (2.5,-0.5);
    \coordinate[label = below: $\ell_e$] (G) at (4,-0.5);
    
    \coordinate[label = below: $x_e^\ast$] (I) at (3,-0.5);
    \coordinate (J) at (3.4,-0.5);
    \coordinate (K) at (3.65,-0.5);
    \coordinate (L) at (3.8,-0.5);
    \coordinate (M) at (3.89,-0.5);
    \coordinate (N) at (3.95,-0.5);
    
    \draw[|-|,very thick] (B) to (C);
    \draw[-|,very thick] (C) to (D);
    \draw[-|,very thick] (D) to (E);
    \draw[-|,very thick] (E) to (F);
    \draw[-,very thick] (F) to (I);
    \draw[|-|,very thick,cyan] (I) to (J);
    \draw[-|,very thick,cyan] (J) to (K);
    \draw[-|,very thick,cyan] (K) to (L);
    \draw[-|,very thick,cyan] (L) to (M);
    \draw[-|,very thick,cyan] (M) to (N);
    \draw[-|,very thick,cyan] (N) to (G);
    
    
\end{tikzpicture}
\caption{Quasi-uniform mesh $\Th^\circ$ for a single pipe (top) and corresponding adaptive mesh $\Th^\eps=\Th^{\eps,1} \cup \color{cyan} \Th^{\eps,2}$ (bottom). The layer region $(x^\ast_e,\ell^e)$ is depicted in cyan.\label{fig:meshes}}
\end{figure}
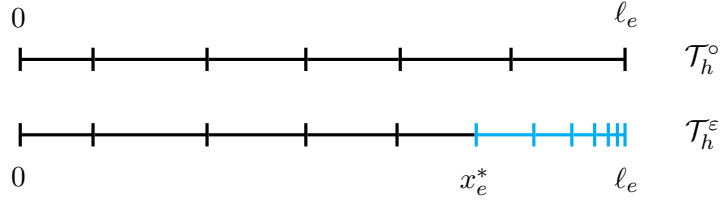
The elements in the layer region, generated by the points $x_e^i$ with $M_e^* \le i \le M_e$, are collected in the set $\Th^{\eps,2}$. 
Using a slight modification of the arguments in \cite{Gartland88,Roos15} and the condition $\eps \ge h^{2k}$, one can show that $|\Th^{\eps,2}| \le C h^{-1}$ with $C$ independent of $h$ and $\eps$; see Section~\ref{sec:proof_main} below.
The complete layer-adapted mesh $\Th^\eps = \Th^{\eps,1} \cup \Th^{\eps,2}$ is then simply obtained by accumulation.

\subsection{Uniform error estimate}

We can now fully describe our adaptive approximation scheme and state the main result of this paper. 
\begin{theorem}\label{thm:error}
Let Assumption \ref{ass:1} hold, and let $(u^\eps,\hat u^\eps)$ be the unique solution of the convection-diffusion problem \eqref{convdiff:1}--\eqref{convdiff:4} with initial value $u^\eps(0)=0$.
Further define 
\begin{align*}
\tilde u_h^\eps := 
\begin{cases} 
    u_h^\circ, & \text{if } \eps < h^{2k},\\
    u_h^\eps, & \text{if } \eps \ge h^{2k},    
\end{cases}
\end{align*}
where $(u_h^\circ,\hat u_h^\circ)$ is the solution of Problem~\ref{prob:convdiff} with $\eps=0$ on the mesh $\Th=\T_h^\circ$, while $(u_h^\eps,\hat u_h^\eps)$ is the solution of Problem~\ref{prob:convdiff} on the layer-adapted mesh $\Th=\T_h^\eps$. 
Then 
\begin{align}\label{eq:error}
    \|u^\eps-\tilde u_h^\eps\|_{L^\infty(0,t_{max};L^2(\E))} \le C \max(h^{k+1},\min(\sqrt{\eps},h^k)).
\end{align}
Moreover, the number of elements in $\T_h^\eps$ can be bounded by $C' h^{-1}$. 
The constants $C,C'$ in these estimates only depend on the bounds in Assumption~\ref{ass:1}, but not on $\eps$ or $h$.
\end{theorem}

\begin{remark}
Since the number of elements $N=|\Th^\eps| \approx C' h^{-1}$, we immediately obtain corresponding bounds 
\begin{align*}
     \|u^\eps-\tilde u_h^\eps\|_{L^\infty(0,t_{max};L^2(\E))} \le C'' \max(N^{-k-1},\min(\sqrt{\eps},N^{-k})).
\end{align*}
A quick look into the error analysis reveals, that one can derive corresponding bounds for the error $\|u^\eps - \tilde u_h^\eps\|_{\eps,h}$ in the mesh- and parameter dependent dG-norm; see \cite{Ern11,Roos15,Thomee07}.
By formally setting $\eps=0$, we obtain the error estimate for the hybrid-dG approximation of the pure transport problem, which was established in~\cite{Egger20}.
\end{remark}
\subsection{Proof of Theorem~\ref{thm:error}}\label{sec:proof_main}

The two cases which are exploited in the construction of the adaptive approximation scheme are treated separately in the following. Note that all constants appearing here and in subsequent proofs only depend on the bounds in Assumption \ref{ass:1}, but not on $\eps$ or $h$. 

\subsection*{Case~1: $\eps < h^{2k}$}
By construction, we have $\tilde u_h^\eps = u_h^\circ$, where $u_h^\circ$ denotes the hybrid-dG approximation of the transport problem. 
Using the triangle inequality, we obtain 
\begin{align*}
\|u^\eps - \tilde u_h^\eps\|_{L^\infty(0,t_{max};L^2(\E))} 
&= \|u^\eps - u_h^\circ\|_{L^\infty(0,t_{max};L^2(\E))} \\
&\le \|u^\eps - u^\circ\|_{L^\infty(0,t_{max};L^2(\E))} + \|u^\circ - u_h^\circ\|_{L^\infty(0,t_{max};L^2(\E))} \\
&\le c \sqrt{\eps} + c' h^{k+1} \le C\max(\sqrt{\eps},h^{k+1}),
\end{align*}
with $(u^\circ,\hat u^\circ)$ being the solution of \eqref{transp:1}--\eqref{transp:4}.
Here, we employed the asymptotic estimate \eqref{eq:asympt_estimate_network} for the continuous solution as well as the error estimate for the hybrid-dG approximation of the limiting transport problem; see \cite{Egger20}. This already yields the bound \eqref{eq:error} for the case $\eps < h^{2k}$.
The assertion about the number of elements is clear, since the mesh $\Th=\Th^\circ$ is quasi-uniform.

\subsection*{Case 2: $\eps\ge h^{2k}$}
%
%
From Lemma~\ref{lemma:prelim_error}, we already know that
\begin{align}\label{eq:prelim_error2}
\|u^\eps(t)-u_h^\eps(t)\|_{L^2(\E)}^2\le c\sum\nolimits_{T\in\T_h}(\eps h_T^{2k} + h_T^{2k+2}) \|u^\e\|_{H^1(0,t_{max};H^{k+1}(T))}^2.
\end{align}
Using the bounds \eqref{eq:deriv}, the choice of the transition point $x^\ast_e$, and noting that $h_T \approx h$ is uniform for elements $T \in \T_h^{\eps,1}$, we find that
\begin{align*}
\sum\nolimits_{T \in \T_h^1} (\eps h_T^{2k} + h_T^{2k+2}) \|u^\eps\|^2_{H^1(0,t_{max};H^{k+1}(T))} 
&\le c' (\eps h^{2k} + h^{2k+2}) \le 2c' h^{2k}. 
\end{align*}
In the layer-adapted part $\T_h^{\eps,2}$ of the mesh, 
the local mesh sizes $h_e^i$ are determined recursively by \eqref{eq:gradmesh1}, 
and the bounds \eqref{eq:deriv} allow us to estimate
\begin{align*}
\sum\nolimits&_{T_e^i\in\T_h^{\eps,2}} \left(\eps (h_e^i)^{2k} + (h_e^i)^{2k+2}\right) \|u^\eps\|_{H^1(0,t_{max};H^{k+1}(T_e^i))}^2\\
&\qquad\qquad
\le c' \sum\nolimits_{T_e^i\in\T_h^2} \left(\eps (h_e^i)^{2k} + (h_e^i)^{2k+2}\right) \int_{x_e^{i-1}}^{x_e^i} \eps^{-2k-2} e^{-2 b_e (\ell_e - x)/\eps}\ dx =: (i) + (ii).
\end{align*}
Using the choice $h_\eps^i = \eps h e^{b_e (\ell_e - x_e^i)/\eps (k+1)}$ in \eqref{eq:gradmesh1}, the first term can be estimated by
\begin{align*}
    (i)
    &
    \le c'\sum\nolimits_{T_e^i\in\T_h^{\eps,2}}\eps^{-1} h^{2k}  \int_{x_e^{i-1}}^{x_e^i}e^{2kb_e(\ell_e-x)/(k+1)} e^{-2 b_e (\ell_e - x)/\eps}\ dx\\
    & 
    =  c' \sum\nolimits_{e\in\E} \eps^{-1} h^{2k} \int_{x_e^\ast}^{\ell_e} e^{-2 b_e (\ell_e - x)/\eps(k+1)}\ dx
    \le c'' h^{2k},
\end{align*}
since the pipe network is finite. 
Similarly, we obtain for the second term
\begin{align*}
    (ii)
    \le c' h^{2k+2} \sum\nolimits_{T_e^i\in\T_h^{\eps,2}} h_e^i 
    \le c' h^{2k+2}\sum\nolimits_{e\in\E}(\ell_e - x_e^\ast) 
    \le c''h^{2k+2}.
\end{align*}
Adding the two estimates for $\T_h^{\eps,1}$ and $\T_h^{\eps,2}$ yields the bound \eqref{eq:error} for $\eps \ge h^{2k}$.

It remains to verify the bound on the number of elements:
The graded mesh $\Th^{\eps,2}$ is of Gartland-type, 
and using the arguments employed in
\cite[p.645]{Gartland88} and \cite[p.8]{Roos15}, one can see that the number of elements in $\Th^{\eps,2}$ intersecting  the intervals $(x^S_e,\ell_e)$, with $x^S_e=\ell_e - \frac{k+1}{b_e} \eps \log(1/h)$ denoting the Shishkin transition point, is bounded by $c' h^{-1}$. 
By assumption, we further have $\eps \ge h^{2k}$, and therefore $|\ell_e -x^\ast_e| \le 2k|\ell - x_e^S|$. Since the mesh gets coarser when moving away from the layer, we conclude that $|\Th^{\eps}| \le C' h^{-1}$ with $C' \le 2k c'$.
%
By a refined analysis, this estimate could even be further improved, i.e., the number of elements intersecting the region $(x^\ast_e,x^S_e)$ is bounded by $c'' \log(1/h)$, which would yield $C'\le c'+c''\log(1/h)h^{-1}$. 
This completes the proof.
\qed

\section{Numerical illustration} \label{sec:num}

For illustration of the flexibility and performance of the proposed hybrid-dG scheme, we now present some numerical results. We first discuss in detail the convergence behaviour for a problem on a single pipe, and then briefly discuss a test problem on a pipe network. 

\subsection{Single pipe}

We consider a single pipe $e=(v_1,v_2) \simeq (0,1)$ of length $\ell = 1$. 
The flow velocity is chosen as $b=1$, and the boundary conditions are described by
\begin{align*}
    \hat g_{v_1}(t) = \tfrac{1}{t_{max}} t^3,\qquad \hat g_{v_2} = 0. 
\end{align*}
The simulations are performed for time $t \le t_{max}=3$, and the initial conditions are given by $u^\eps(0)=0$. 
This choice of the problem data satisfies Assumption~\ref{ass:1} with $m=2$, which suffices to guarantee optimal convergence rates for polynomial order $k \le 2$.

\textbf{Discretization and error estimation.}
For our numerical tests, we utilize the proposed hybrid-dG method with piecewise quadratic finite elements, and we set $\alpha=1$ for the stabilization parameter. 
The convergence rates of Theorem~\ref{thm:error} apply with $k=2$. 
%
%
For the time integration, we use the Radau~IIA Runge-Kutta method with $3$ stages and an uniform time step $\tau$. This scheme can be interpreted as a discontinuous Galerkin method with second order polynomials, and the time-discretization errors can be shown to be of order $O(\tau^{3})$; see~\cite{Akrivis11,Thomee07}. We choose $\tau = h/2$, such that the effect of the time discretization can be considered negligible. 
%
To estimate the discretization errors, we compute a reference solution $u_{ref}^\eps$ on a mesh $\Th^{ref}$, which is obtained by two uniform refinements of the computational mesh $\Th$ used in our analysis.
In addition, the time step $\tau^{ref}=\tau/4$ is reduced accordingly.
The actual discretization error is then estimated by
\begin{align}
   \|u^\eps - \tilde u_h^\eps\|_{ref} := \max\nolimits_{n=0,..,N_{ref}}\| (u^\eps_{ref})(t^n) - (I_{ref} \tilde u_h^\eps)(t^n)\|_{
    L^2(\E)},
\end{align}
where $I_{ref}$ denotes the interpolation operator onto the reference mesh and $N_{ref}$ the number of time steps used for the computation of the reference solution. 

\begin{figure}[ht!]
    \centering
    \includegraphics[scale=0.55]{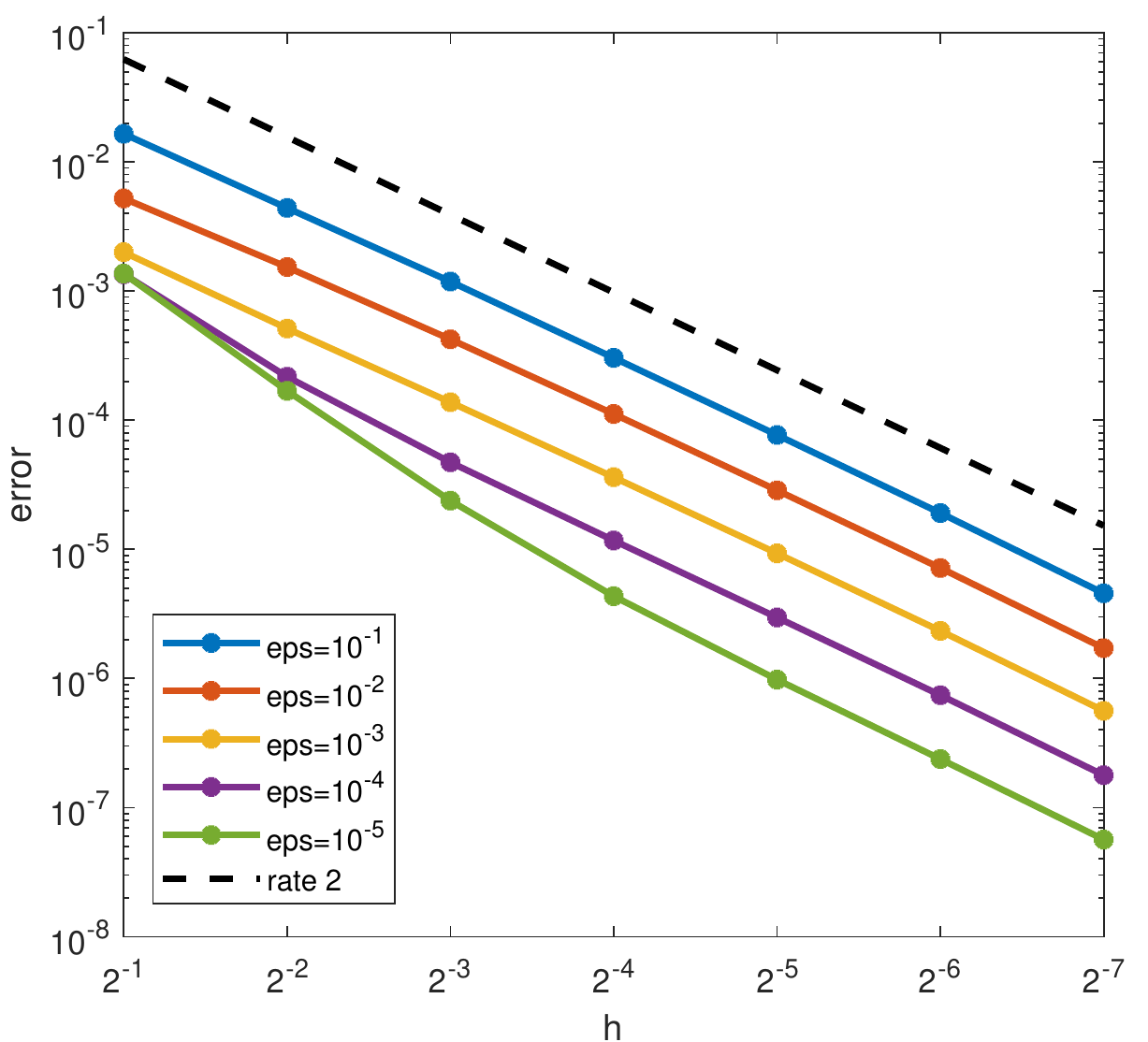}
    \hspace*{1em}
    \includegraphics[scale=0.55]{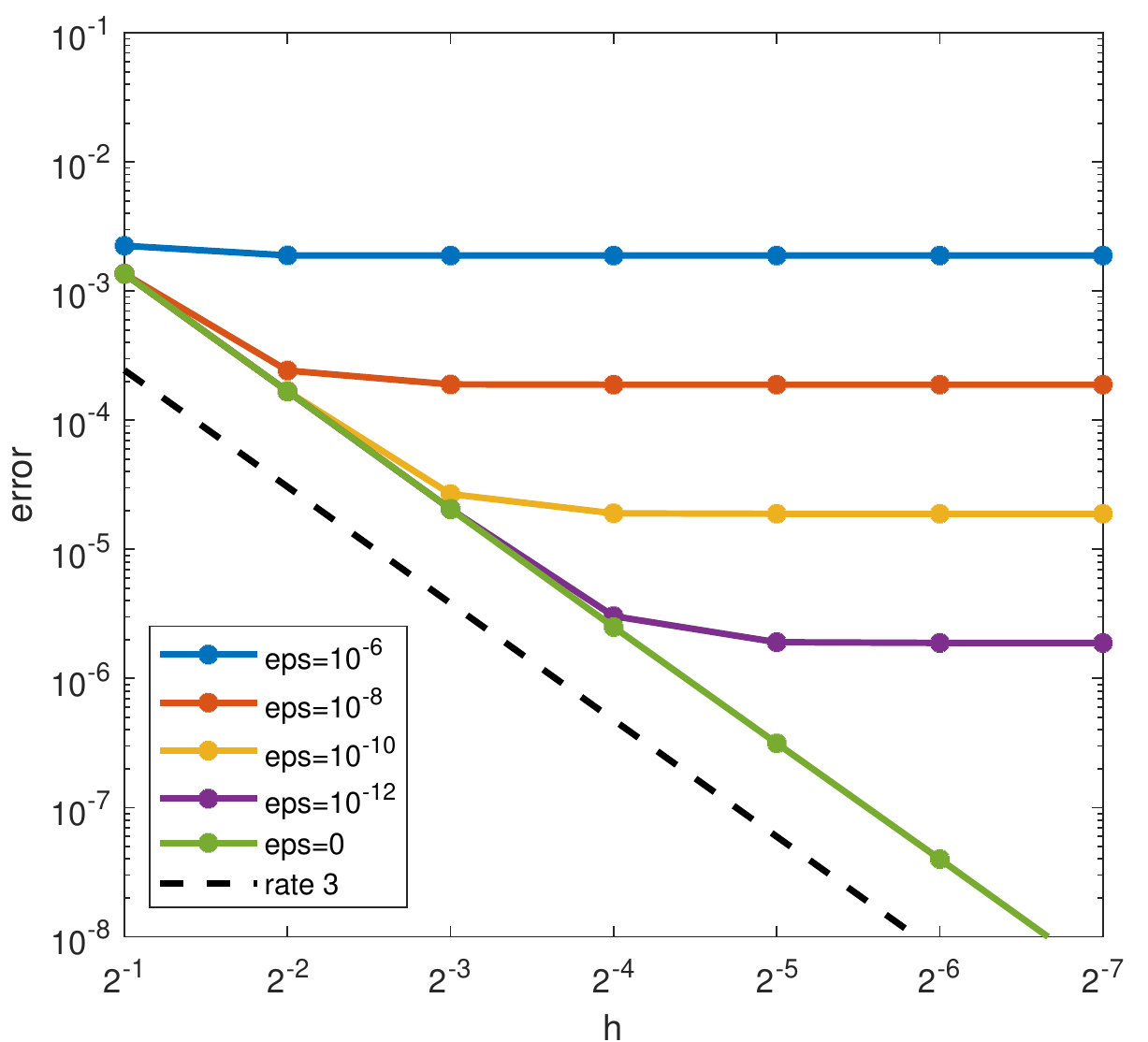}
    \caption{(Left) Error for $\tilde u_h^\eps = u_h^\eps$ on graded mesh $\T_h^\eps$. (Right) Error for $\tilde u_h^\eps = u_h^\circ$ on $\T_h^\circ$.}
    \label{fig:error}
\end{figure}

\textbf{Results.}
In the left plot of Figure~\ref{fig:error}, we display the numerical errors $\|u^\eps - \tilde u_h^\eps\|_{ref}$ obtained by the hybrid-dG method on the graded mesh $\T_h^\eps$ for different values of $\eps>0$. This is the relevant error in the diffusive regime $\eps \ge h^{2k}$.
In accordance with Theorem~\ref{thm:error}, we observe second order convergence. 
On coarse meshes and for small $\eps$, the diffusive terms can be considered as a perturbation of the pure transport problem, which explains the increase in the convergence rate for the test with $\eps=10^{-5}$.
In our tests, the number of elements $|\Th^{\eps,2}|$ in the boundary layer $(x^\ast,\ell)$ is approximately $3\cdot h^{-1}$, and only few elements lie in the layer $(x^\ast,x^S)$ outside the Shishkin point; see Section~\ref{sec:proof_main}.
%
In the right part of Figure~\ref{fig:error}, we plot the errors $\|u^\eps - u_h^\circ\|_{ref}$, which are relevant when $\eps < h^{2k}$. 
From the proof of Theorem~\ref{thm:error}, one can see that $\|u^\eps - u_h^\circ\|_{ref} \le C \, \max(h^{k+1},\sqrt{\eps})$. This leads to a saturation when $\eps>h^{2k+2}$, i.e., $\sqrt{\eps}$ becomes the dominating term in the error estimate, which is the behaviour observed in the tests.
%
%
In summary, the numerical results are in perfect agreement with the theoretical predictions.

\subsection{A pipe network}

As a second test problem, we consider a pipe network consisting of $11$ edges and $11$ vertices, with $3$ entries, $2$ exits and $1$ loop; see Figure \ref{fig:gaslib11} for a sketch. 
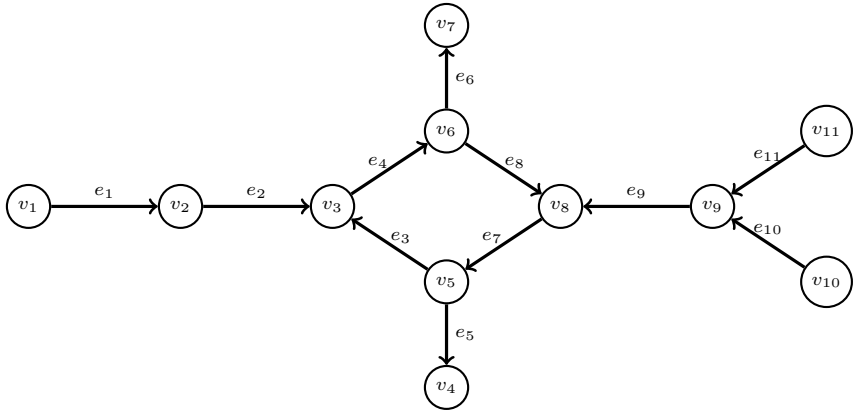
\begin{figure}[ht!]
\centering
\begin{tiny}
\begin{tikzpicture}[scale=2]
\node (A) at (0,0) [circle,draw,thick] {$v_1$};
\node (B) at (1,0) [circle,draw,thick] {$v_2$};
\node (D) at (2,0) [circle,draw,thick] {$v_3$};
\node (G) at (2.75,-0.5) [circle,draw,thick] {$v_5$};
\node (E) at (2.75,0.5) [circle,draw,thick] {$v_6$};
\node (F) at (2.75,1.2) [circle,draw,thick] {$v_7$};
\node (H) at (2.75,-1.2) [circle,draw,thick] {$v_4$};
\node (I) at (3.5,0) [circle,draw,thick] {$v_8$};
\node (J) at (4.5,0) [circle,draw,thick] {$v_9$};
\node (K) at (5.25,0.5) [circle,draw,thick] {$v_{11}$};
\node (L) at (5.25,-0.5) [circle,draw,thick] {$v_{10}$};
\draw[->, very thick] (A) to node[above] {$e_1$} (B);
\draw[->, very thick] (D) to node[above left=-0.1cm] {$e_4$} (E);
\draw[->, very thick] (G) to node[right] {$e_5$} (H);
\draw[->, very thick] (E) to node[right] {$e_6$} (F);
\draw[->, very thick] (E) to node[above right=-0.1cm] {$e_8$} (I);
\draw[->, very thick] (I) to node[above left=-0.1cm] {$e_7$} (G);
\draw[->, very thick] (K) to node[above] {$e_{11}$} (J);
\draw[->, very thick] (L) to node[above] {$e_{10}$} (J);
\draw[->, very thick] (B) to node[above] {$e_2$} (D);
\draw[->, very thick] (J) to node[above] {$e_9$} (I);
\draw[->, very thick] (G) to node[above right=-0.1cm] {$e_3$} (D);
\end{tikzpicture}
\end{tiny}
\caption{Topology of the GasLib-11 network, taken from \cite{gaslib}.}
\label{fig:gaslib11}
\end{figure}
For ease of presentation, we set $\ell_e=1$ for the length of all edges. 
%
The volume flow rates are given by
\begin{align*}
    b_{e_1} = b_{e_2} = b_{e_5} = b_{e_6} = b_{e_9} = 2,\quad
    b_{e_3} = b_{e_8} = b_{e_{10}} = b_{e_{11}} = 1,\quad
    b_{e_4} = b_{e_7} = 3.
\end{align*}
By this choice, condition \eqref{eq:b} is satisfied.
As boundary conditions we choose
\begin{align*}
    \hat g_{v_1}(t) = \frac{2}{t_{max}^3} t^3,\quad \hat g_{v_4}(t) = \hat g_{v_7}(t) = 0,\quad \hat g_{v_{10}}(t) = \frac{3}{2 t_{max}^4} t^4,\quad \hat g_{v_{11}}(t) = \frac{5}{2 t_{max}^3} t^3.
\end{align*}
The time horizon is set to $t_{max}=6$, and the initial conditions are again $u^\eps(0)=u^\circ(0)=0$.
Like in the previous test, Assumption~\ref{ass:1} is satisfied with smoothness parameter $m=2$.
%

\textbf{Results.}
At the vertices $v_3$, $v_8$, and $v_9$, which have two in-going and one out-going pipe, we expect discontinuities in the concentration field in the transport limit $\eps=0$, and corresponding internal layers for the convection--diffusion problem with $\eps>0$. Furthermore, we expect boundary layers at the outflow vertices $v_4$ and $v_7$; compare with the plot in Figure~\ref{fig:network}.  
\begin{figure}
\centering
\includegraphics[scale=0.65]{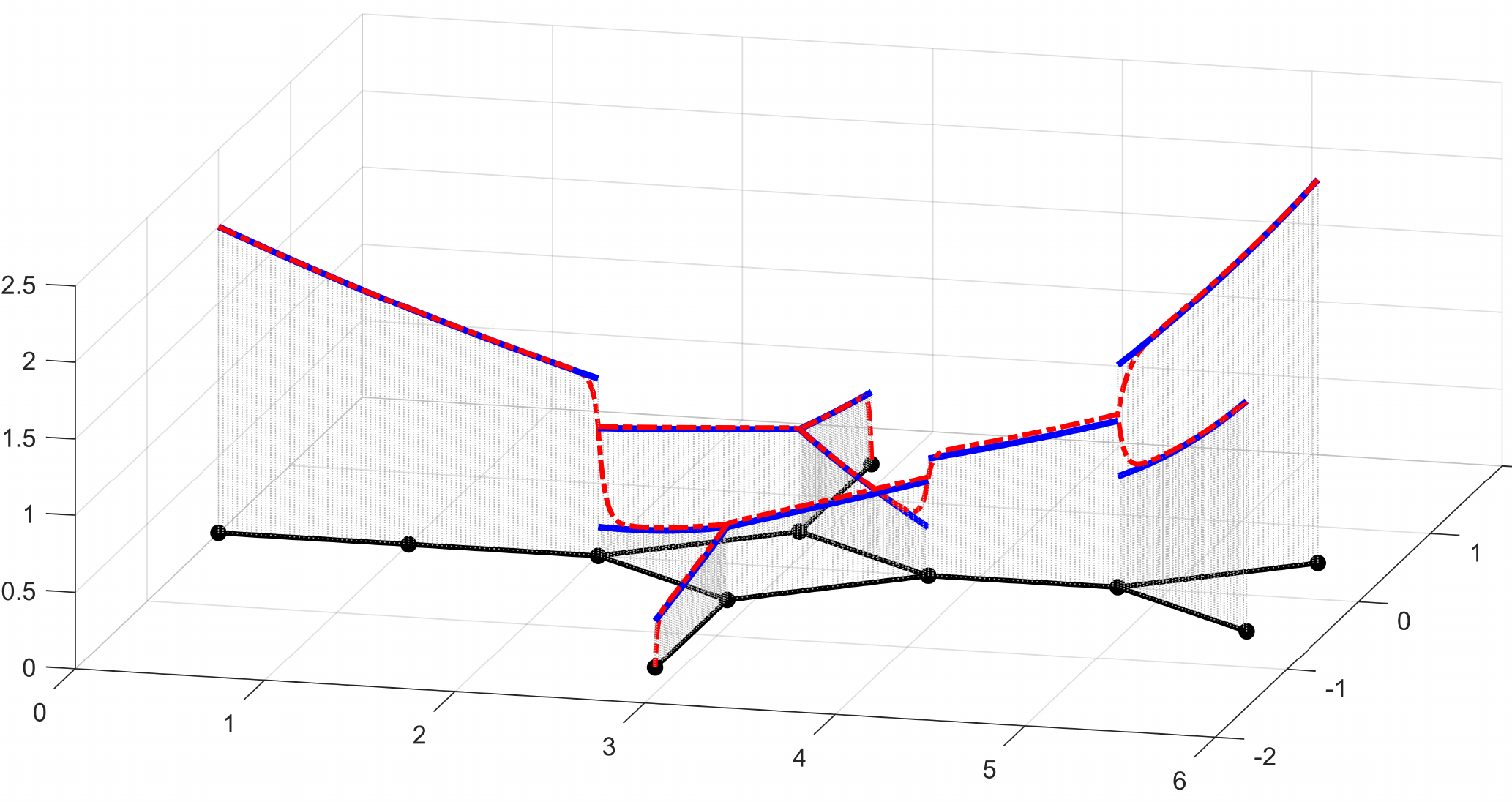}
\caption{Solution to the convection-diffusion problem for $\eps=0.05$ (red,dashed) and the limiting transport problem for $\eps=0$ (blue, solid). 
Discontinuities appearing for the transport limit are smoothed out for $\eps>0$ by diffusion, leading to boundary layers at the outflow boundaries and internal junctions.
Away from the layers, the solutions of the two different problems can hardly be distinguished.}
\label{fig:network}
\end{figure}
Using the same discretization strategy and test setup as explained for the case of a single pipe, we repeated the convergence tests and observed exactly the same convergence behavior as depicted in Figure \ref{fig:error} for the case of a single pipe. Since no additional insight is obtained from these results, we omit their presentation here.


\section{Discussion} \label{sec:disc}

In this paper, we studied the numerical approximation of singularly perturbed parabolic convection-diffusion problems on one-dimensional pipe networks by a hybrid discontinuous Galerkin method. 
A key feature of this method is the automatic handling of the coupling conditions at network junctions and boundary conditions at network boundaries, whose number and type changes in the vanishing diffusion limit $\eps \to 0$. 
Together with an upwind treatment of the convective terms, the proposed scheme is asymptotic preserving. 
To avoid possible instabilities resulting from boundary and internal layers, geometrically adapted meshes of Gartland-type were employed in the layer regions.

In the transport regime, i.e., when $\eps < h^{2k}$, we proposed to use the numerical approximation of the pure transport problem. This allowed us to choose the transition point for the layer adapted mesh as $x^*(\eps) \approx \ell - \eps \log(1/\eps)$, while still guaranteeing quasi-optimal number $N \approx h^{-1}$ of elements. Moreover, this choice allowed us to resort on standard localized discretization error estimates for dG methods. 
The numerical results demonstrate the validity and sharpness of our estimates.

For lowest order approximations with $k=1$, we observed convergence of order $O(h^{k+1})$ even in the diffusion dominated regime. 
While the use of hybrid variables in our discretization method is particularly useful for the automatic handling of the coupling conditions at network junctions, alternative discretization strategies, e.g., standard dG schemes, upwind finite differences, or SUPG-Galerkin methods, may be used for the approximation along the pipes. 
The main steps of our analysis should carry over to such schemes almost verbatim. 
Also the additional consideration of time discretization seems possible without major complications.
A theoretical investigation of these topics is left for future research.

\section*{Acknowledgements}
The authors are grateful for financial support by the German Research Foundation (DFG) via grant TRR~154, subproject~C04, project-number 239904186.



\newpage
\appendix

\section*{Appendix}

For completeness of the presentation, we now give the proofs for some auxiliary results, which were used in our error analysis and follow by standard arguments.

\section{Proof of the bounds \eqref{eq:deriv} in Lemma~\ref{lem:deriv}}
\label{sec:thm:deriv}

Let $u^\eps$ be the solution of \eqref{convdiff:1}--\eqref{convdiff:4} with initial value $u^\eps(0)=0$. We want to show that
\begin{align}\label{eq:derivatives}
|\dt^n\dx^j u^\eps_e(x,t)| \leq C\, (1+\e^{-j} e^{-b_e(\ell_e-x)/\eps})
\end{align}
for all
$n\le m,\, j\le 2(m-n)+1$; recall that $m$ is the regularity index of Assumption~\ref{ass:1}.
For establishing these bounds, we will use the following weak maximum principle for convection-diffusion problems on networks; see \cite[Lemma 7]{EggerCD20}.

\begin{lemma} \label{lemma:max}
Let $u\in C^1([0,t_{max}];L^2(\E))\cap C^0([0,t_{max}];H^1(\E) \cap H^2_{pw}(\E))$ satisfy 
\begin{alignat*}{4}
\dt u_e+b_e \dx u_e - \eps \dxx u_e &\ge 0, \qquad && \forall e\in\E,\\
\sum\nolimits_{e\in\E(v)} \eps \dx u_e(v) n_e(v) &=0,  \qquad && \forall v\in\V_0,\\
u(v) &\ge 0, \qquad  && \forall v\in\V_\partial,
\intertext{for all $0 < t < t_{max}$ with initial conditions}
u(0) &\ge 0, \qquad && \forall e\in\E.
\end{alignat*}
Then, the function $u$ is non-negative, i.e., $u \ge 0$ on $\E$ for all $t\in[0,t_{max}]$. 
\end{lemma}

The above estimates can now be established by induction over $j$ and $n$.
%
For $j,n=0$ the claim follows directly from Lemma~\ref{lemma:max} with the usual comparison arguments: 
We define $w_e(x,t):=\max_t|\hat g(t)| \pm u^\eps_e(x,t)$, which satisfies all conditions of Lemma \ref{lemma:max}, and therefore is non-negative.
This implies that $|u^\eps|$ can be bounded by the maximum norm of the boundary data $\hat g$. 
By linearity and time-invariance of the equations, $(\dt^n u^\eps, \dt^n \hat u^\eps)$ again solves \eqref{convdiff:1}--\eqref{convdiff:4}, but with boundary data $\dt^n \hat g$; further note that $\dt^n u^\eps(0)=0$ for $n \le m$. With the same reasoning as above we thus obtain the bounds for $|\dt^n u^\eps|$, $n\le m$.
%

Induction over $j$:
Assume that \eqref{eq:derivatives} holds for all $0 \le i\le j-1$ and all $n \le m$. 
%
In a first step, we verify that $\dt^{n}\dx^{j} u^\eps_e(\ell_e)\leq c\,\e^{-j}$ for all $e\in\E$. By the mean value theorem, we know that there exists $y\in (\ell_e-\eps,\ell_e)$, such that 
\begin{align*}
\dt^{n}\dx^{j} u^\eps_e(y) = \tfrac{1}{\eps} \big(\dt^{n}\dx^{j-1} u^\eps_e(\ell_e)-\dt^{n}\dx^{j-1} u^\eps_e(\ell_e-\eps)\big) \leq c\, \eps^{-j},
\end{align*}
where we used the induction hypothesis in the last step.
Using \eqref{convdiff:1}, we further see that
\begin{align}\label{eq:deriv1}
    \dt (\dt^n\dx^j u^\eps_e) = -b_e \dx (\dt^n\dx^j u^\eps_e) + \eps \dxx (\dt^n\dx^j u^\eps_e).
\end{align}
By the fundamental theorem of calculus and the induction hypothesis, 
we conclude that
\begin{align}\label{eq:deriv2}
\dt^{n}\dx^{j} u^\eps_e(\ell_e) =&\ \dt^{n}\dx^{j} u^\eps_e(y) 
+ \int_y^{\ell_e} \dxx \dt^{n}\dx^{j-1} u^\eps_e(x)\ dx \\
=&\ \dt^{n}\dx^{j} u^\eps_e(y) 
+\!\! \int_y^{\ell_e}\!\!\! \tfrac{1}{\eps}\big(\dt^{n+1}\dx^{j-1} u^\eps_e(x)+b_e \dt^{n}\dx^{j} u^\eps_e(x)\big)\, dx\nonumber\\ 
\le&\ c\,\eps^{-j} 
+\max\nolimits_{x\in[y,\ell_e]}|\dt^{n+1}\dx^{j-1} u^\eps_e(x)
+\tfrac{1}{\eps}b_e \dt^{n}\dx^{j-1} u^\eps_e(x)\big|
\le c'\,\eps^{-j}. \nonumber
\end{align}
%
Let us now fix an arbitrary $t \in [0,t_{max}]$ and set $w(x):=\dt^n\dx^ju^\eps(x,t)$. 
Then by \eqref{eq:deriv1}, the function $w$ solves the ordinary differential equation
\begin{align}
b_e w(x) - \eps w'(x) = \eta(x) := - \dt^{n+1}\dx^{j-1} u^\eps_e(x,t) \label{eq:deriv3}
\end{align}
with terminal value $w(\ell_e) = \dt^n\dx^{j}u^\eps(\ell_e,t)$.
Using the induction hypothesis, the right hand side of this problem can be estimated by
\begin{align}\label{eq:deriv4}
    |\eta(x)|\le c(1 + \e^{-(j-1)} e^{-b_e(\ell_e-x)/\eps}).
\end{align}
Expressing the solution of \eqref{eq:deriv3} via the variation-of-constants-formula, we find that
\begin{align*}
w(x) =&\ w(\ell_e) e^{-b_e(\ell_e - x)/\eps} 
+\tfrac{1}{\eps} \int_x^{\ell_e} e^{-b_e(\sigma-x)/\eps} \eta(\sigma)\ d\sigma\\
\le&\ c'\,\eps^{-j} e^{-b_e(\ell_e-x)/\eps} 
+ \tfrac{1}{\eps}\int_x^{\ell_e} e^{-b_e(\sigma-x)/\eps} c(1+\eps^{-(j-1)} e^{-b_e(\ell_e-\sigma)/\eps})\ d\sigma\\
\le&\ c'\,\eps^{-j} e^{-b_e(\ell_e-x)/\eps} + 
c\eps^{-j} e^{-b_e(\ell_e-x)/\eps}(\ell - x) + \tfrac{c}{b_e}(1-e^{-b_e(\ell - x)/\eps})\\
\le&\ c (1+\eps^{-j} e^{-b_e(\ell_e-x)/\eps}).
\end{align*}
Here we employed \eqref{eq:deriv2} and \eqref{eq:deriv4} in the subsequent estimates. 
This yields the bounds \eqref{eq:deriv} for index $i=j$ and $n \le m$ and, by induction, concludes the proof of Lemma~\ref{lem:deriv}. \qed

\section{Basic properties of the hybrid-Dg scheme}
\label{sec:properties}

We now establish discrete stability, well-posedness, and consistency of the discretization scheme in Problem~\ref{prob:convdiff}.
We start with showing ellipticity of the governing bilinear forms.

\begin{lemma}\label{lemma:stab}
Let $b_h, d_h$ be as in \eqref{def:bh}--\eqref{def:dh}. 
Then for all $w_h\in W_h$ and $\hat{w}_h\in\hat{W}_h$, we have 
\begin{align}
b_h(w_h,\hat{w}_h;w_h,\hat{w}_h)&=\tfrac{1}{2}| b^{1/2}(w_h-\hat{w}_h)|_{\dT_h}^2,\label{bh:stab}\\
d_h(w_h,\hat{w}_h;w_h,\hat{w}_h)&= \|\dx w_h^2\|_{\T_h}^2 + |(\tfrac{\alpha}{h_{loc}})^{1/2}(w_h-\hat w_h)|_{\dT_h}^2\label{dh:stab}
\end{align}
with $h_{loc}|_T=h_T$ for $T\in\T_h$.
\end{lemma}

\begin{proof}
Let $T=(x^{in},x^{out})$ be one of the elements of the mesh $\Th$. Then, in accordance with the notation introduced in Section~\ref{subsec:not}, we call $x^{in}$ the inflow and $x^{out}$ the outflow boundary of $T$, and we denote by ${\dT_h^{in}}$ and ${\dT_h^{out}}$ the collections of all inflow and outflow boundaries of elements $T \in \T_h$. 
Equation \eqref{bh:stab} then follows from
\begin{align*}
b_h(w_h,\hat{w}_h,w_h,\hat{w}_h)
=&\ -(b \, w_h,\dx w_h)_{\T_h}+\langle nb \, \wup,w_h-\hat{w}_h\rangle_{\dT_h}\\
=&\ -\tfrac{1}{2}\langle nb \, w_h,w_h\rangle_{\dT_h} 
+ \langle nb \, w_h,w_h-\hat w_h\rangle_{\dT_h^{out}}
+ \langle nb \,  \hat w_h,w_h-\hat w_h\rangle_{\dT_h^{in}}
\\
=&\ \tfrac{1}{2}|b^{1/2} w_h|_{\dT_h} - 
\langle b \, w_h,\hat{w}_h\rangle_{\dT_h} +
\tfrac{1}{2}|b^{1/2} \hat{w}_h|_{\dT_h}^2
= \tfrac{1}{2}|b^{1/2} (w_h - \hat{w}_h)|_{\dT_h}^2.
\end{align*}
Here we used that $|b^{1/2}\hat{w}_h|_{\dT_h^{in}}=| b^{1/2}\hat{w}_h|_{\dT_h^{out}}$ due to the conservation condition \eqref{eq:b} on the flow rates, and the fact that $\hat w_h^v = 0$ for $v\in\V_\partial$.
Equation \eqref{dh:stab}, on the other hand, follows directly, since the second and third term in \eqref{def:dh} cancel each other.
\end{proof}

As a direct consequence, we obtain the well-posedness of the discretization scheme.
\begin{lemma}
Let Assumption~\ref{ass:1} hold. Then Problem~\ref{prob:convdiff} has a unique solution $$
(u_h^\eps,\hat u_h^\eps) \in C^{1}([0,t_{max}];W_h) \times C^0([0,t_{max}];\hat W_h).
$$
\end{lemma}
\begin{proof}
From the previous lemma, we can deduce that the combined bilinear form $b_h + \eps d_h$ is elliptic on the discrete spaces $W_h \times \hat W_h$. 
The hybrid variables $\hat u_h^\eps$ can therefore be eliminated from the discrete problem on the algebraic level, leading to an ordinary differential equation for $u_h^\eps$ alone.
Existence of a unique solution and its regularity then follow by the Picard-Lindelöf theorem and elementary arguments. 
\end{proof}

As a next ingredient for our analysis, we verify consistency of the approximation scheme.

\begin{lemma}\label{lemma:consistency}
Let $(u^\eps,\hat u^\eps)$ be the solution of \eqref{convdiff:1}--\eqref{convdiff:4} with initial value $u^\eps(0)=0$. 
Further define $\hat u^\eps(x) = u^\eps(x)$ for $x \in \Xh$ and set $\hat u^\eps(v)=0$ for $v \in \dV$. 
Then
\begin{align*}
    (\dt u^\eps(t),w_h)_{\T_h}+b_h(u^\eps(t),\hat{u}^\eps(t);w_h,\hat{w}_h)
    +\eps d_h(u^\eps(t),\hat{u}^\eps(t);w_h,\hat{w}_h)
    &= \ell_h^\eps(t;w_h)
\end{align*}
for all $w_h\in W_h$, $\hat w_h \in \hat W_h$, and all $0 \le t \le t_{max}$, i.e., the method is consistent.
\end{lemma}

\begin{proof}
Let us first test the bilinear form $d_h$ with $w_h\in W_h$ and $\hat{w}_h \equiv 0$, which yields
\begin{align*}
d_h(u^\eps,\hat{u}^\eps;w_h,0)
&=(\dx u^\eps,\dx w_h)_{\T_h}-\langle n\dx u^\eps,w_h\rangle_{\dT_h} \\
& \qquad \qquad 
 +\langle n (u^\eps-\hat{u}^\eps),\dx  w_h\rangle_{\dT_h}  
+ \langle \tfrac{\alpha}{h_{loc}}(u^\eps-\hat{u}_\e),w_h\rangle_{\dT_h} \\
&=-(\dxx u^\eps,\dx w_h)_{\T_h}
+ \langle n u^\eps, \dx w_h\rangle_{\dV} + \langle \tfrac{\alpha}{h_{loc}} u^\eps,w_h\rangle_{\dV}.
\end{align*}
Here we used integration-by-parts on every element for the first term, whose boundary contributions cancels the second term.  
Since $u^\eps(x)=\hat u^\eps(x)$ for all $x \in \Xh \cup V_0$, the contributions of the third and fourth term vanish at internal mesh points; further note that $\hat u^\eps(v)=0$ on the network boundary $v \in \dV$. 
In a similar manner, we observe that 
\begin{align*}
b_h(u^\eps,\hat{u}^\eps;w_h,0)
&=-(b \, u^\eps,\dx w_h)_{\Th}
 +\langle b n \, u^{up}, w_h\rangle_{\dT_h}
= (b\dx u^\eps,w_h)_{\T_h}
\end{align*}
for all $w_h\in W_h$, since $ nb \, u^{up}(v) := \max(nb,0) u^\eps(v) + \min(nb,0) \hat u^\eps(v) = nb \, u^\eps(v)$ by continuity of $u^\eps$ across junctions.  
Using \eqref{convdiff:1} and \eqref{convdiff:2}, we then see that 
\begin{align*}
(\dt u_h^\eps,w_h)_{\Th} + b_h(u_h^\eps, \hat u_h^\eps; w_h,0) + \eps d_h(u_h^\eps,\hat u_h^\eps;w_h,0) &= \ell_h(w_h)    
\end{align*}
for all $w_h \in W_h$ and $0 \le t \le t_{max}$. 
The continuity and coupling conditions \eqref{convdiff:3}--\eqref{convdiff:4}, on the other hand, imply validity of the variational identities for $\hat w_h \in \hat W_h$ when testing with $w_h\equiv 0$.
In summary, we thus obtain consistency of the method. 
\end{proof}

\section{Proof of Lemma~\ref{lemma:prelim_error}}
\label{sec:proof_prelim}

Based on consistency and discrete stability, we can now prove the local error estimate~\eqref{eq:prelim_error}.
Following \cite[Chapter~12]{Thomee07}, we define a projection operator $\pi_h:H_{pw}^{1}(\E)\rightarrow W_h$ by
\begin{alignat}{2}
\pi_h w^-(x_i^{e}) &= w^-(x_i^{e})\qquad &&\text{for all}\ i=1,\dots,M^{e},\ e\in\E, \label{def:proj1}\\
\int_{T}(w-\pi_h w)p\ dx &= 0\qquad &&\text{for all}\ p\in P_{k-1}(T),\ T\in\T_h \label{def:proj2}
\end{alignat}
with up- and downwind value of $w$ at some point $x$ denoted by
\begin{align*}
    w^-(x) = \lim_{s\nearrow 0} w(x+s),\qquad w^+(x) = \lim_{s\searrow 0} w(x+s).
\end{align*}
The following properties of the projection $\pi_h$ can be found in \cite[App.C]{John16}.
\begin{lemma}\label{lemma:projerror}
The operator $\pi_h : H^1_{pw}(\Th) \to W_h$ is a well-defined projection. Moreover, for any element  $T=(x^{in},x^{out})\in \T_h$ and $w\in H_{pw}^{k+1}(\E)$, we have $\pi_h w(x^{out})=w^-(x^{out})$ and 
\begin{align}
\|w-\pi_h w\|_{L^2(T)} &\le C h_{T}^{k+1} \|w\|_{H^{k+1}(T)}, \\
\|\dx w-\dx \pi_h w\|_{L^2(T)} &\le C h_{T}^k \|w\|_{H^{k+1}(T)},\\
|w^+(x^{in})-\pi_h w^+(x^{in})| &\le C h_{T}^{k+1/2} \|w\|_{H^{k+1}(T)},\\
|\dx w-\dx \pi_h w|_{\partial T}&\le Ch_{T}^{k-1/2} \|w\|_{H^{k+1}(T)}.
\end{align}	
\end{lemma}
%
By the triangle inequality, we can now split the discretization error 
\begin{align*}
\|u^\eps(t)-u_h^\eps(t)\|_{L^2(\E)}\leq \|\eta_h(t)\|_{L^2(\E)} + \|e_h(t)\|_{L^2(\E)}
\end{align*}
into a projection error $\eta_h:=u^\eps-\pi_h u^\eps$ and a discrete error component $e_h:=u_h^\eps-\pi_h u^\eps$.
Via the estimate of Lemma~\ref{lemma:projerror} and the continuous embedding of $H^1(0,t_{max}) \subset L^\infty(0,t_{max})$, we can bound the projection error by
\begin{align*}
\|\eta_h\|_{L^\infty(0,t_{max};L^2(\E))}^2
& \le c \sum\nolimits_{T\in\T_h}\!\! h_T^{k+1}\| u^\eps\|_{H^1(0,t_{max};H^{k+1}(T))}^2.
\end{align*}
The remaining part of this section is now devoted to the estimation of the discrete error. 
We denote by $\hat \pi_h : H^1(\E) \to \hat W_h$, $u \mapsto u|_{\Xh \cap \V_0}$ the interpolation of the continuous function $u$ at the interior mesh points. This coincides with the definition  of $\hat u^\eps$ in Lemma~\ref{lemma:consistency}, and consequently $\hat \eta_h := \hat u^\eps - \hat \pi_h u^\eps=0$. 
From the consistency of the scheme stated in Lemma~\ref{lemma:consistency}, one can deduce that \begin{align*}
\tfrac{1}{2} \tfrac{d}{dt} \|e_h\|^2 
= (\dt e_h, e_h)_{\Th} 
&=-b_h(e_h,\hat{e}_h;e_h,\hat{e}_h)-\eps d_h(e_h,\hat{e}_h;e_h,\hat{e}_h)
+(\partial_t\eta_h,e_h)\\
& \qquad \qquad +b_h(\eta_h,\hat{\eta}_h;e_h,\hat{e}_h)
+ \eps d_h(\eta_h,\hat{\eta}_h;e_h,\hat{e}_h) \\
&= (i) + (ii) + (iii) + (iv) + (v).
\end{align*}
With the properties of Lemma~\ref{lemma:stab}, we see that 
\begin{align*}
    (i) + (ii) = -\eps\|\dx e_h\|_{\T_h}^2 - \eps|(\tfrac{\alpha}{h_{loc}})^{1/2}(e_h-\hat e_h)|_{\dT_h}^2
-\tfrac{1}{2}| b^{1/2}(e_h-\hat{e}_h)|_{\dT_h}^2.
\end{align*}
By Cauchy-Schwarz and Young inequalities, we further obtain
\begin{align*}
    (iii) = (\dt u^\eps - \dt \pi_h u^\eps,e_h)
    \le c \sum\nolimits_{T\in\T_h} h_T^{2k+2} \|\dt u^\eps\|_{H^{k+1}(T)}^2 + \tfrac{1}{2}\|e_h\|^2.
\end{align*}
Here, we used the fact that $\dt\pi_h u^\eps = \pi_h \dt u^\eps$ as well as the projection error estimates of Lemma~\ref{lemma:projerror}.
For the fourth term we observe that
\begin{align*}
    (iv) = -(b\eta_h,\partial_x e_h)_{\T_h}
+\langle nb \, \eta_h^{up},e_h\rangle_{\dT_h} = 0,
\end{align*}
where the first term vanishes due to \eqref{def:proj2} and the second one due to \eqref{def:proj1} and the definition of $\hat\pi_h$, which together yield $\eta_h^{up}=0$.
Cauchy-Schwarz and Young's inequality as well as a discrete trace inequality finally allow to estimate the last term by
\begin{align*}
(v) 
\le \eps\|\dx e_h\|^2 + \eps|(&\tfrac{\alpha}{h_{loc}})^{1/2}(e_h-\hat e_h)|^2_{\dT_h}
+ \tfrac{\eps}{2}\|\dx \eta_h\|^2\\
&+ \tfrac{\eps}{2}|(\tfrac{h_{loc}}{\alpha})^{1/2}\dx\eta_h|^2_{\dT_h} + \tfrac{\eps}{2}(\alpha+C_{tr}^2)|h_{loc}^{-1/2}(\eta_h-\hat \eta_h)|^2_{\dT_h}.
\end{align*}
Note that the first two terms cancel with the two last terms in the estimate of $(i)+(ii)$. 
The remaining terms can again be bounded by the projection error estimates of Lemma~\ref{lemma:projerror}, 
which finally leads to 
\begin{align*}
    \tfrac{1}{2} \tfrac{d}{dt} \|e_h\|^2 \le c h^{2k+2} \|\dt u^\eps\|^2_{H^{k+1}_{pw}(\T_h)} +  c' \eps h^{2k} \|&u^\eps\|_{H^{k+1}_{pw}(\T_h)}^2
    + \tfrac{1}{2}\|e_h\|^2.
\end{align*}
By application of Gronwall's lemma, we thus obtain
\begin{align*}
    \|e_h\|_{L^\infty(0,t_{max};L^2(\E))}^2 \le c'' \sum\nolimits_{T\in\T_h}(\e h_T^{2k} + h_T^{2k+2}) \|u^\eps\|_{H^1(0,t_{max};H^{k+1}(T))}^2,
\end{align*}
where we used that $e_h(0)=\pi_h u(0) - u_h(0)=0$ by definition of the initial values. 
This already concludes the proof of Theorem~\ref{lemma:prelim_error}. 
\qed

\end{document}